\documentclass[12pt,twoside]{amsart}

\usepackage{amssymb,amsmath}
\theoremstyle{plain}
\newtheorem{theorem}{Theorem}[section]
\newtheorem{lemma}[theorem]{Lemma}

\newtheorem{corollary}[theorem]{Corollary}

\theoremstyle{definition}
\newtheorem{remark}[theorem]{Remark}
\newtheorem{defn}[theorem]{Definition}

\numberwithin{equation}{section}

  \input xypic
  \xyoption{all}

 \usepackage{tikz}
\usetikzlibrary{arrows}

 \usepackage{times}

  \usepackage{amsmath, amsthm, amsfonts}

\def \F{\mathcal F}

\def \mc{\mathcal}
\def \inv{^{-1}}

\def \real{\mathbb{R}}

\def \mff{\mathsf}
\def \indexn{^{(k)}}

\def \half{\frac{1}{2}}

\def \b{\bar }

\def \p{\partial}

\def\t{\mathfrak{t}}

              \def \F{{\mathcal F}}

\def \p{\partial}

\def \v {\vskip 0.1in}
\def \n {\noindent}

\def \F{\mathcal F}

\def \mc{\mathcal}
\def \inv{^{-1}}

\def \v{\vskip 0.1in}
\def \n{\noindent}

\def \real{\mathbb{R}}

\def \indexn{^{(k)}}

\def \half{\frac{1}{2}}

\def \b{\bar }

\def \p{\partial}

\def \indexm {_{k} }

\def\t{\mathfrak{t}}

\begin{document}

\title{Some Estimates for a Generalized Abreu's Equation }
\author[Li]{An-Min Li}
\address{Department of Mathematics,
Sichuan University,
 Chengdu, 610064, China}
\email{anminliscu@126.com}
\author[Lian]{Zhao Lian}
\address{Department of Mathematics,
Sichuan University,
Chengdu, 610064, China}
\author[Sheng]{Li Sheng}
\address{Department of Mathematics,
Sichuan University,
Chengdu, 610064, China}
\email[Corresponding author]{lshengscu@gmail.com}
\thanks{ Li acknowledges the support of NSFC Grant NSFC11521061. \\
${}\quad\ $Sheng acknowledges the support of NSFC Grant NSFC11471225.}
{\abstract
We study a generalized Abreu equation and derive some estimates.
\endabstract
}

\maketitle
 
\n



\section{Introduction}\label{Sec-Intro}
 This is one of a sequence of   papers, aiming at generalizing the results of Chen, Li and Sheng to homogeneous toric bundles. In this paper we establish some estimates.
\v

Let $\mathfrak{t}$ be a linear space of dimension $n$, $\mathfrak{t}^*$ be its dual space. We identify $\mathfrak{t}$ and $\mathbb{R}^n$ with coordinates $x=(x^1,...,x^n)$, and  identify $\mathfrak{t}^*$ and $\mathbb{R}^n$ with coordinates $\xi=(\xi_1,...,\xi_n)$. Let $\Omega^*\subset \mathfrak{t}^*$ be a bounded convex domain, $u(\xi_1,...,\xi_n)$ be a $C^{\infty}$ strictly convex function defined in $\Omega^*$ satisfying the following nonlinear fourth-order partial differential equation
\begin{equation}\label{eqn 1.1}
\frac{1}{\mathbb{D}}\sum_{i,j=1}^n\frac{\partial^2 \mathbb{D} u^{ij}}{\partial \xi_i\partial \xi_j}=-A.
\end{equation}
Here, $\mathbb{D}(\xi)>0$ and $A(\xi)$ are two given smooth functions defined on $\overline{\Omega}^*$ and $(u^{ij})$ is the inverse of the Hessian matrix $(u_{ij})$.

The equation \eqref{eqn 1.1} was introduced by Donaldson \cite{D5} and Raza \cite{R}
in the study of the scalar curvature of toric fibration, see also \cite{N-1}. We call \eqref{eqn 1.1} a generalized Abreu equation. In this paper we derive some estimates for $u$, which will be used in our next works in the study of prescribed scalar curvature problems on homogeneous toric bundles.
\v

\section{Preliminaries}\label{sec-Prliminaries}



\smallskip

Let $f=f(x)$ be a smooth and strictly convex function defined in a
convex domain  $\Omega\subset \mathfrak{t}$.
As $f$ is strictly
convex, $G_f$ defined by
$$
G_f=\sum_{i,j}\frac{\partial^{2}f}{\partial
x^{i}\partial x^{j}}d x^{i} dx^{j}= \sum_{i,j} f_{ij}d x^{i} dx^{j}
$$
is a Riemannian metric in $\Omega$. The gradient of $f$
defines a (normal) map $\nabla^f$ from $\mathfrak{t}$ to $\mathfrak{t}^*$:
$$
\xi=(\xi_1,...,\xi_n)=\nabla^f(x) =\left(\frac{\partial f}{\partial
x^{1}},...,\frac{\partial f}{\partial x^{n}}\right).
$$
The function $u$ on $\mathfrak{t}^*$
$$
u(\xi)=x\cdot\xi - f(x)
$$
is called the {\it Legendre transform} of $f$. We write
$$u=L(f),\;\;\; \Omega^\ast=\nabla^f(\Omega)\subset \mathfrak{t}^*.$$
Conversely, $f=L(u).$ It is well-known that $u(\xi)$ is a smooth and strictly convex function.
Corresponding to $u$, we have the metric
$$G_u=\sum_{i,j}\frac{\partial^{2}u}{\partial
\xi_{i}\partial \xi_{j}}d \xi_i d\xi_j=\sum_{i,j} u_{ij}d\xi_id\xi_j.$$
Under the normal map $\nabla^f$, we have
\begin{align*}
\tfrac{\partial \xi_i}{\partial x^{k}}=\tfrac{\partial^2 f}{\partial x^{i}\partial x^{k}},\;\;\;\;
\left(\tfrac{\partial^2 f}{\partial x^{i}\partial x^{k}}\right)=\left(\tfrac{\partial^2 u}{\partial \xi_i\partial \xi_k}\right)^{-1},\;\;\det\left(\tfrac{\partial^2 f}{\partial x^{i}\partial x^{k}}\right)
=\det\left(\tfrac{\partial^2 u}{\partial \xi_i\partial \xi_k}\right)^{-1}.\end{align*} Then,
$$(\nabla^f)^*(G_u)=\sum_{i,j}\frac{\partial^{2}f}{\partial
x^{i}\partial x^{j}}d x^{i} dx^{j}=G_f,$$
i.e., $ \nabla^f: (\Omega, G_f)\to (\Omega^*,
G_u) $ is locally isometric.
\v\n
Let $ \rho=\left[\det(f_{ij})\right]^{-\frac{1}{n+2}}, $ we
introduce the following affine invariants:
\begin{equation}\label{eqn_2.1}
\Phi=\frac{\|\nabla\rho\|^2_{G_f}}{\rho^2}=\frac{1}{(n+2)^2}\|\nabla \log\det(u_{ij})\|_{G_{u}} \end{equation}
\begin{equation}\label{eqn_2.2}
4n(n-1)J=\sum f^{il}f^{jm}f^{kn}f_{ijk}f_{lmn}= \sum
u^{il}u^{jm}u^{kn}u_{ijk}u_{lmn},
\end{equation}
\begin{equation}\label{eqn_2.3}
\Theta = J + \Phi.
\end{equation}
In particular, let $\Omega^*=\Delta\subset \mathfrak{t}^*$ be a Delzant polytope, $M$ be the associate toric variety. Set
$$w^i= x^{i} + \sqrt{-1}y^i,\;\;z^i=e^{w^i/2}.$$
Let $\omega_{f}$ be a K\"{a}hler metric with a local potential function $f$. The Ricci curvature
and the scalar curvature of $\omega_{f}$ are given by
$$
R_{i\b j}= - \frac{\partial^{2}}{\partial z^{i}\partial \bar
z^{j}} \left(\log \det\left(f_{k\bar l}\right)\right), \;\;\;\mc
S=\sum f^{i\bar j}R_{i\b j},
$$
respectively. When we use the log-affine coordinates,   the Ricci curvature
and the scalar curvature can be written as
$$
R_{i\b j}= - \frac{\partial^{2}}{\partial x^{i}\partial x^{j}}
\left(\log \det\left(f_{kl}\right)\right), \;\;\;\mc S=-\sum
f^{ij}\frac{\partial^{2}}{\partial x^{i}\partial x^{j}} \left(\log
\det\left(f_{kl}\right)\right).
$$
Define
\begin{equation}\label{eqn_1.1c}
\mc K\;=\;\|Ric\|_f +\|\nabla Ric\|_f
^{\frac{2}{3}}+\|\nabla^2 Ric\|_f
^{\frac{1}{2}}.
\end{equation}
Put
$$
R_{ij}= - \frac{\partial^{2}}{\partial x^{i}\partial x^{j}}
\left(\log \det\left(f_{kl}\right)\right), \;\;\;\mc S=-\sum
f^{ij}\frac{\partial^{2}}{\partial x^{i}\partial x^{j}} \left(\log
\det\left(f_{ij}\right)\right).$$
In terms of $\xi$ and $u(\xi)$ the scalar curvature can be written as
\begin{equation}\label{eqn 2.4}
\sum_{i,j=1}^n\frac{\partial^2 u^{ij}}{\partial \xi_i\partial \xi_j}=-\mathcal{S}.
\end{equation}
\eqref{eqn 2.4} is called the Abreu's Equation.
\v\n
For a given smooth function $\mathbb{D}(\xi)>0$ defined on $\overline{\Omega}^*$ we define an operator $\mathcal{S}_{\mathbb{D}}$ as
\begin{equation}\label{eqn 2.5}
\mathcal{S}_{\mathbb{D}}(u)=-\frac{1}{\mathbb{D}}\sum_{i,j=1}^n\frac{\partial^2 \mathbb{D} u^{ij}}{\partial \xi_i\partial \xi_j} .
\end{equation}
 Let $A(\xi)$ be given smooth function defined on $\overline{\Omega}^*$, we study the generalized Abreu equation.
\begin{equation}\label{eqn 2.7}
\frac{1}{\mathbb{D}}\sum_{i,j=1}^n\frac{\partial^2 \mathbb{D} u^{ij}}{\partial \xi_i\partial \xi_j}=-A.
\end{equation}
Set
$$
\mathbb{F}:=\frac{\mathbb D}{\det(u_{ij})},\;\;\;U^{ij}= \det(u_{kl})u^{ij}.$$
Through the (normal) map $\nabla u$ we can view $\mathbb{D}$ and $\mathbb{F}$ to be smooth functions on the toric variety $M$.
\begin{lemma}\label{lemma_2.6} The generalized Abreu equation \eqref{eqn 2.7} is equivalent to any of the following  equations:
\begin{equation}\label{eqn 2.8}
\sum_{i,j} U^{ij}\mathbb F_{ij}  =-\mathbb D A,
\end{equation}
\begin{equation}\label{eqn 2.9}
\sum_{i,j} u^{ij}\frac{\p^2(\log \mathbb F)}{\p \xi_i \p \xi_j} +\sum_{i,j}u^{ij}\frac{\p (\log \mathbb F)}{\p \xi_i}\frac{\p (\log \mathbb F)}{\p \xi_j} = -A,
\end{equation}
\begin{equation}\label{eqn 2.10}
\mathcal{S}_{\mathbb{D}}(f):=-\sum_{i,j} f^{ij}\frac{\p^2(\log \mathbb F)}{\p x^{i} \p x^{j}} -\sum_{i,j}f^{ij}\frac{\p (\log \mathbb F)}{\p x^{i}}\frac{\p (\log \mathbb D)}{\p x^{j}}=A.
\end{equation}
\end{lemma}
\begin{proof}  Since $\sum_{i} U^{ij}_{i}=0,$ the equation \eqref{eqn 2.7} can be written as \eqref{eqn 2.8} and \eqref{eqn 2.9}. Note that
\begin{align*}
\sum u^{ij}\frac{\p ^2 }{\p \xi_{i} \p \xi_j}&=\sum f^{ij}\frac{\p ^2 }{\p x^{i} \p x^{j}}-\sum f^{ij}\frac{\p}{\p x^{i}}\log \det(f_{kl})\frac{\p}{\p x^{j}} \\
& = \sum f^{ij}\frac{\p ^2 }{\p x^{i} \p x^{j}}-\sum f^{ij}\frac{\p\log \mathbb F }{\p x^{i}}\frac{\p}{\p x^{j}}+\sum f^{ij}\frac{\p\log \mathbb D }{\p x^{i}} \frac{\p}{\p x^{j}}.
\end{align*}
We can re-write \eqref{eqn 2.9} in the coordinates  $(x^{1},\cdots, x^{n})$ as \eqref{eqn 2.10}.
\end{proof}

Let $\omega_g$ be the Guillemin metric with local potential function ${\bf g}$. For any $T^n$-invariant metric $\omega\in [\omega_g]$ with local potential function ${\bf f}$, there is a function $\phi$ globally defined on $M$ such that
$${\bf f}={\bf g} + \phi.$$
Set
$$
\mc C^\infty(M,\omega_g)=\{\mathsf f|\mathsf f=\mathsf g+\phi,\; \phi\in C^\infty_{\mathbb T^{2}}(M) \mbox{ and }\omega_{\mathsf g+\phi}>0\}.
$$
  Fix a large constant $K_o>0$. We set
$$
\mc C^\infty(M,\omega_g;K_o)
=\{\mathsf{f}\in \mc C^\infty(M,\omega_g)|
|\mc S_{\mathbb{D}}(f)|\leq K_o\}.
$$
Choose a local coordinate system $z^1,...,z^n$. Denote
$$
H:=\frac{\det(g_{i\bar{j}})}{\det(f_{i\bar{j}})}
.$$
It is known that $H$ is a global function defined on $M$.
Set
$$\mathbb{H}:=\frac{\mathbb{F}_g}{\mathbb{F}_f}=
\frac{\mathbb{D}_g}{\mathbb{D}_f}H,$$
where
$$\mathbb{D}_g:=(\nabla^g)^*\mathbb{D},\;\;\;\mathbb{D}_f:=(\nabla^f)^*\mathbb{D},\;\;\;\mathbb{F}_{g}=\mathbb{D}_g\det(g_{ij}),\;\;\;\mathbb{F}_{f}=\mathbb{D}_f\det(f_{ij}).$$

\section{Upper bound of H}

 In this section we assume that $\Omega^*=\Delta\subset \mathbb \mathfrak{t}^*$ is a Delzant polytope, and $M$ is the associate toric variety. Let $\mu:M\rightarrow \bar{\Delta}\subset \mathfrak{t}^*$ be the moment map. We introduce notations:
$$R_g= \max\limits_{M}\left\{|\sum g^{ij}\left(\log \mathbb{F}_g\right)_{ij}|+|\nabla_{x}\log\mathbb{F}_g|^2\right\},$$
$$\mathcal{D}=\max\limits_{\bar{\Delta}}\left\{|\nabla_{\xi}\log \mathbb D |,\right\}\;\;\;\mathcal{R}:=\max\left\{ R_g, \mathcal{D}^2,(diam(\Delta))^2\right\},$$
where $\left|\nabla_{x} \log\mathbb  F_{g}\right|^2 =\sum\left(\frac{\p \log\mathbb  F_{g}}{\p x^{i}}\right)^2$ and  $ \left|\nabla_{\xi} \log\mathbb  D\right|^2 =\sum\left(\frac{\p \log\mathbb  D}{\p\xi_{i}}\right)^2$.   We prove
\begin{theorem}\label{theorem 3.1}  For any $\phi\in C^{\infty}(M,\omega_{g})$ we have
\begin{equation}\label{eqn_3.2}
H \leq C \exp
\left\{(2\mathcal{R}+1)(\max_M\{\phi\}-\min_M\{\phi\})\right\},
\end{equation}
where $C$ is a constant depending only on $n$, $\max|\mathcal{S}_{\mathbb{D}}(f)|$ and $\mathcal R$. Here  we denote $\mathcal{S}_{\mathbb{D}}(f)=\mathcal{S}_{\mathbb{D}}(\phi)$.
\end{theorem}
\v\n
{\bf Proof.} Consider the function
$$\F := \exp\{-C\phi\}\mathbb{H},$$
where $C$ is a constant to be determined later. $\F$ attains its
maximum at a point $p^*\in M$. We have, at $p^*$,
\begin{equation}\label{eqn_3.3}
-C \frac{\p\phi}{\p z^i} + \frac{\p \log \mathbb{H}}{\p z^{i}}=0,\end{equation}
\begin{equation}\label{eqn_3.4}
- C f^{i\bar j} \phi_{i\bar j}+f^{i\bar j}(\log \mathbb{H})_{i\bar j}\leq 0.\end{equation}
If $\mu(p^*)\in \Delta$ we use the affine log coordinates. Note that for any smooth function $F$ depending only on $x$ we have
 $$
 \frac{1}{2} \frac{\p F}{\p x^{i}} = \frac{\p F}{\p w^{i}} = \frac{\p F}{\p z^{i}}\frac{\p z^{i}}{\p w^{i}} =\frac{ z^{i} }{2} \frac{\p F}{\p z^{i}} ,
 $$
If $\mu(p^*)\in \p \Delta,$ we can choose $q^{*}$ such that $\mu(q^*)\in \Delta$ very close to $p^*$ such that, at $q^{*}$,
\begin{equation}\label{eqn_3.5}
|-C \phi_{i} + (\log \mathbb{H})_{i}|\leq \epsilon,\end{equation}
\begin{equation}\label{eqn_3.6}
- C f^{ij} \phi_{ij}+f^{ij}(\log \mathbb{H})_{ij}\leq \epsilon.\end{equation}
where $F_{i}=\frac{\p F}{\p x^{i}},F_{ij}=\frac{\p^2 F}{\p x^{i}\p x^{j}},etc.$ From \eqref{eqn_3.5} we get
\begin{equation}\label{eqn_3.7}
|(\log \mathbb F_f)_{i}-(\log \mathbb F_g)_{i} + C(f-g)_{i}|\leq \epsilon.
\end{equation}
Inserting \eqref{eqn 2.10} and \eqref{eqn_3.7} into \eqref{eqn_3.6} we have
\begin{align}\label{eqn_3.2}
&Cf^{ij}g_{ij}-Cn + A+ \sum_{i,j}f^{ij} (\log \mathbb D_f)_{i}(\log \mathbb F_g)_{j}+\sum f^{ij}\left(\log \mathbb{F}_g\right)_{ij} \\
&-C\sum_{i,j}\left|f^{ij} (\log \mathbb D_f)_{i}(f-g)_{j}\right| \leq \epsilon (1+\sum_{i,j}\left|f^{ij} (\log \mathbb D_f)_{i}\right| ). \nonumber
\end{align}
By an orthogonal transformation we can choose the
coordinates $x^{1},...,x^{n}$ such that, at $q^*$,
$$f_{ij}= \lambda_i \delta_{ij},\;\;\;g_{ij}= \mu_i
\delta_{ij}.$$
Note that
\begin{align*}
&|\sum f^{ij} (\log \mathbb D_f)_{i}|=\left|\frac{\p }{\p \xi_j}(\log \mathbb D)\right|\leq \mathcal{D},\\
&|2\sum_{i,j}f^{ij} (\log \mathbb D_f)_{i}(\log \mathbb F_g)_{j}|= \left|2\sum_{i,j} \frac{\p\log \mathbb D}{\p \xi_j}(\log \mathbb F_g)_{j}\right| \leq 2\mathcal R
\end{align*}
 $$\left|\frac{\p }{\p x^{j}}(f-g)\right|\leq diam(\Delta),\;\;\;
|\sum g^{ij}\left(\log \mathbb{F}_g\right)_{ij}|\leq \mathcal{R}.$$
From \eqref{eqn_3.2}, we get
$$C\left( \tfrac{\mu_1}{\lambda_1} + \cdot\cdot\cdot +
\tfrac{\mu_n}{\lambda_n}\right) - Cn + A - \left(
\tfrac{\mu_1}{\lambda_1} + \cdot\cdot\cdot +
\tfrac{\mu_n}{\lambda_n}\right) \mathcal{R}-C_1 \leq  \epsilon(1+n\mathcal{D})$$
for some constant $C_1=3\mathcal R.$
We choose $C = 2\mathcal{R}+1$ and
apply an elementary inequality
$$
\frac{1}{n}\left(\frac{\mu_1}{\lambda_1} + \cdot\cdot\cdot +
\frac{\mu_n}{\lambda_n}\right)\geq \left(\frac{\mu_1\cdots\mu_n}{\lambda_1\cdots
\lambda_n}\right)^{1/n}
$$
to get
$$n(\mathcal{R}+1) H^{\tfrac{1}{n}} \leq n(2\mathcal{R}+1) +  |A|+C_{1},$$ where we use the arbitrariness of $\epsilon.$
As $\mathbb{D}$ is bounded below and above on $\bar{\Delta}$,
it follows that, at $p^*$,
$$\exp\{-C\phi\}\mathbb{H} \leq C\left(2+\frac{|A|+3 \mathcal R }{n(\mathcal{R}+1)}\right)^n\exp\{-(2\mathcal{R}+1)\min_M\{\phi\}\}.$$
for some constant $C>0$. Then \eqref{eqn_3.3} follows. $\Box$
\v
Donaldson \cite{D4} derived a $L^{\infty}$ estimate for the Abreu's equation and $n=2$. His method can be applied directly to the
generalized Abreu Equation ( see also \cite{N-1}).
\begin{theorem}\label{theorem_3.2}
Let $n=2$ and $\Delta\subset \mathfrak{t}^{*}$ be a Delzant polytope, $\mathbb{D}>0$ and $A$ be two smooth functions defined on $\bar\Delta$. Let $u\in C^{\infty}(\Delta,v)$ satisfying \eqref{eqn 1.1}. Suppose that $\Delta$ is $(\mathbb{D},A,\lambda)$-stable. Then there is a constant
$\mff C_o>0$, depending on $\lambda$, $\Delta$, $\mathbb{D}$ and $\|\mathcal{S}_{\mathbb{D}}(u)\|_{C^0}$, such that $|\max\limits_{\bar \Delta} u-\min\limits_{\bar \Delta} u|\leq \mff C_o$.
\end{theorem}
\v
Combining Theorem \ref{theorem 3.1} and Theorem \ref{theorem_3.2} we get
\begin{theorem}\label{theorem 3.3} Let $n=2$ and $\Delta\subset \mathfrak{t}^{*}$ be a Delzant polytope, $M$ be the associate toric variety. Given two smooth functions $\mathbb{D}>0$ and $A$ defined on $\bar{\Delta}$. Let $u\in \mathbf{S}$ be a solution of the generalized   Abreu's Equation \eqref{eqn 1.1}. Suppose that $\Delta$ is $(\mathbb{D}, A, \lambda)$ stable. Then there is a constant $\mff C_1>0$ depending on $diam(\Delta)$, $\mathbb{D}$ and $\lambda$, such that the following holds everywhere on $M$
\begin{equation}
H\leq \mff C_1.\end{equation}
\end{theorem}
\v
\section{Estimates of the determinant}
\label{sect_4}

\begin{defn}\label{definition_3.1.1}
A convex domain $\Omega\subset R^{n} $ is called normalized when its center of mass is 0 and $n^{-\frac{3}{2}}D_1(0) \subset  \Omega  \subset
D_1(0).$
\end{defn}

\begin{lemma}\label{lemma_4.1}
Let $\Omega\subset \mathbb{R}^2$ be a bounded
convex domain, and $f$ be a smooth and strictly convex function defined on $\Omega$ with $$f|_{\partial \Omega}=C,\;\;\;\;\inf_{\Omega}f=0.$$ Assume that $f$ satisfies the equation \eqref{eqn 2.10}.
Suppose that
\begin{equation}\label{eqn 4.1}
\sup_{\Omega}|\nabla f|\leq \mff N_{1},\;\;\;\;\;\;\; \sup_{\Omega}\left|\nabla_{\xi} \log\mathbb  D\right|\leq \mff N_{1}
\end{equation}
for some constant $\mff N_{1}>0.$   Then the following
estimate holds:
\begin{equation}\label{eqn 4.2}
\exp\left\{-\tfrac{4C}{C-f}\right\}\det(f_{ij})\leq \mff C_{2},\;\;\;\forall x\in \Omega,
\end{equation}
where $\mff C_{2}$ is a constant depending only on $C$, $\mff N_{1}$ and
 $\max_{\Omega} A$.
\end{lemma}
\v\n
{\bf Proof.}  Consider the function
\begin{equation}
 F:=\exp\left\{-\frac{4C}{C-f}+\epsilon\sum \left(\frac{\p f}{\p x^{i}}\right)^2\right\}\mathbb F .
 \end{equation}
  Clearly, $F$
attains its supremum at some interior point $p^*$. At
$p^*$, we have,
\begin{align}
&-g f_{i}+2\epsilon\sum f_{j}f_{ij}+(\log \mathbb F)_{i}=0 \label{eqn_4.4} \\
& -2g -g'\sum f^{ij}f_{i}f_{j}+2\epsilon\sum f_{ii} +2\epsilon\sum f_{i}f^{kl}f_{kli}\nonumber\\
&+\sum f^{ij}(\log \mathbb F)_{ij} \leq 0.\label{eqn 4.5}
\end{align}
where $g=\frac{4C}{(C-f)^2}.$ By \eqref{eqn_4.4}, at
$p^*$, we have
\begin{align*}
2\epsilon\sum f_{i}f^{kl}f_{kli}&=2\epsilon\sum f_{i}\frac{\p\log \det(f_{kl})}{\p x^{i}}=2\epsilon\sum f_{i}\frac{\p\log \mathbb F }{\p x^{i}}- 2\epsilon\sum f_{i}f_{ij}\frac{\p\log \mathbb D}{\p \xi_{j}} \\
&=2\epsilon g\sum f^2_{i}-4\epsilon^2\sum f_{ij}f_{i}f_{j} - 2\epsilon\sum f_{i}f_{ij}\frac{\p\log \mathbb D}{\p \xi_{j}} ,
\end{align*}
and
\begin{align*}
\sum f^{ij}(\log \mathbb F)_{ij}&=-A -\sum f^{ij}(\log \mathbb F)_{i}(\log \mathbb D)_{j}=-A -\sum  (\log \mathbb F)_{i}\frac{\p \log \mathbb D }{\p \xi_{i}} \\
&= -A+2\epsilon f_{i}f_{ij}\frac{\p\log \mathbb D}{\p \xi_{j}}-\sum gf_{i}\frac{\p\log \mathbb D}{\p \xi_{i}}.
\end{align*}
Inserting these into \eqref{eqn 4.5}, using \eqref{eqn 4.1}, choosing $\epsilon<\frac{1}{4\mff N_{1}^2}$,  we have
\begin{align*}
-2g-\mff N_{1}^2g'\sum f^{ii}+\epsilon \sum f_{ii}-A-\mff N_{1}^2 g
  \leq 0.
\end{align*}
Let $\lambda_1,\lambda_{2}$ be the eigenvalues of $(f_{ij})$. Then we have
\begin{equation}
-2g-\mff N_{1}^2 g'\frac{\lambda_1+\lambda_2}{\lambda_1\lambda_{2}}+\epsilon(\lambda_1+\lambda_{2})-A- \mff N_{1}^2g \leq 0.
\end{equation}
Since $\frac{\lambda_{1}+\lambda_{2}}{2}\geq\sqrt{\lambda_1\lambda_2}$, one can easily obtain that
\begin{equation}
\epsilon\lambda_{1}\lambda_{2}\leq (2g+\max_{\Omega} A +\mff N_{1}^2 g )\sqrt{\lambda_{1}\lambda_{2}}+2\mff N_{1}^2g'.
\end{equation}
 By $\exp\{-\frac{4C}{C-f}\}(g^2+g')\leq C_{2}$ and Cauchy inequalities we have
$$
F\leq C_{3}
$$
where $C_{3}>0$ is a constant depending only on $\mff N_{1}$ and $\max_{\Omega} A$.
Then the  lemma follows. $\Box$
\v
By the same method in \cite{CLS5} we can prove the following two lemmas. For reader's convenience we sketch the proofs here.

\begin{lemma}\label{lemma_4.2} Let $\Omega^*\subset \mathbb{R}^{n}$ be a normalized  convex domain,
and $\mathbb{D}>0$, $A$ be two smooth functions defined on $\bar\Omega^*$.
Let $u$ be a smooth and strictly convex function defined in $\Omega^*$ with $$u|_{\partial \Omega^*}=C,\;\;\;\;\inf_{\Omega^*}u=u(p)=0.$$
Let $f$ be the Legendre transformation of $u$. Assume that $u$ satisfies the generalized Abreu equation equation \eqref{eqn 2.7}.
Then there is a constant $d\geq 1$, independent of $u$, such that
\begin{equation}\label{eqn_2.1}
\exp \left\{-\frac{4C}{C-u }\right\} \frac{\det
(u_{ij})}{(d+f)^{2n}}\leq \mff C_{3}
\end{equation}
for some constant $\mff C_{3}>0$ depending only on $n$,$max|A|$ and $C$.
\end{lemma}
\v\n
{\bf Proof.} We can show as in \cite{CLS5} that there are constants $d\geq 1$, $b>0$ such that
$$\frac{\sum (x^{k})^2}{(d+f)^{2}}\leq b.$$
Consider the following function
$$
F = \exp \left\{-\frac{m}{C-u } + L\right\}\frac{1}{\mathbb{F}(d+f)^{2n}},$$
where
$$
L = \epsilon \frac{\sum (x^{k})^2}{(d+f)^{2}}.
$$
$m$ and $\epsilon$ are positive constants to be determined later.
Clearly,
F attains its supremum at some interior point $p^*$ of $\Omega^*$.
We have, at $p^*$,
\begin{equation}\label{eqn_2.2}
F_{i} = 0,
\end{equation}
\begin{equation}\label{eqn_2.3}
\sum u^{ij}F_{ij} \leq 0,
\end{equation}
where we denote $F_i = \frac{\partial F}{\partial \xi_i}$, $F_{ij} =
\frac{\partial^2 F}{\partial \xi_i\partial \xi_j},f_i =
\frac{\partial f}{\partial \xi_i}$ and so on. Using the PDE \eqref{eqn 2.8} we calculate both
expressions \eqref{eqn_2.2} and \eqref{eqn_2.3} explicitly:
\begin{equation}\label{eqn_2.4}
-\frac{m}{(C-u)^2}u_i + L_i - 2n\frac{f_i}{d+f} - (\log \mathbb{F})_i = 0,
\end{equation}
and
\begin{eqnarray}\label{eqn_2.5}
&&-\frac{2m}{(C-u)^3}\sum u^{ij}u_iu_j - \frac{mn}{(C-u)^2} + \sum u^{ij}{L_{ij}}
\\
&&- 2n \frac{\sum u^{ij}f_{ij}}{d+f}
 + 2n\frac{\sum u^{ij}f_if_j}{(d+f)^2} + \frac{\sum u^{ij}\mathbb{F}_i\mathbb{F}_j}{\mathbb{F}^2} + A \leq 0.
 \nonumber
\end{eqnarray}
We choose  $\epsilon =
\frac{1}{8000n^2b}$, $m= 4C$. Note that $\frac{1}{C_{1}}\leq \mathbb{D}\leq C_{1}$ for some constant $C_{1}>0$. By the same calculation in \cite{CLS5}  we get
$$ \exp \left\{-\frac{m}{C-u } + \epsilon \frac{\sum (x^{k})^2}{(d+
f)^{2}}\right\}\frac{1}{(d+f)^{2n}\mathbb{F}}\leq d_1$$ for some constant
$d_1
>0$ depending on $n$, $b$ and $\mathbb D$. Since $F$ attains its
maximum at $p^*$, the Lemma follows. $\Box$
\v
Let $\Delta\subset \mathbb R^2$ be the Delant polytope defined by
$$
\Delta=\{\xi|l_i(\xi)> 0,\;\;\; 0\leq
i\leq d-1\}
$$
where $l_i(\xi):=\langle\xi,\nu_i\rangle- \lambda_i$ and $\nu_i$ is
the inward pointing normal vector to the edge  $\ell_i$ of $\Delta$.
  Set $v=\sum_{i} l_i\log l_i$ and
$$
\mc C^\infty(\Delta,v)=\{u| u=v+\psi \mbox{ is strictly convex, }
\psi\in C^\infty(\bar \Delta)\}.
$$

\v
\begin{lemma}\label{lemma_2.3.a}
 Let $\Delta\subset \mathbb R^2$ be a Delzant ploytope, $\mathbb{D}>0$ and $A$
be smooth functions on $\bar\Delta$. 
 Suppose that $0\in \Delta^{o}$, and $u\in \mathcal C^{\infty}(\Delta,v)$ satisfies the generalized Abreu equation
\eqref{eqn 2.7} and 
$$u(0)=\inf u,\;\;\;\;\;\nabla u(0)=0.$$
Then,
$$\frac{\det( u_{ij})}{(d+f)^4}(p) \leq \mff C_{4} (d_E(p, \partial
\Delta))^{4},$$
where $\mff C_{4}$ is a positive constant depending $diam(\Delta)$, $\max_{\bar\Delta}\mathbb D$,
$\min_{\bar\Delta}\mathbb D$ and $\|A\|_{L^{\infty}(\Delta)}$. Here $d_{E}(p,\p \Delta)$ denotes the Euclidean distance from $p$ to $\p\Delta.$
 \end{lemma}
 \v\n
{\bf Proof.}   We can find constants $d\geq 1$, $b>0$ such that
$$\frac{\sum (x^{k})^2}{(d+f)^{2}}\leq b.$$
Consider the following function
$$
F = (r^2-\gamma)^2\exp \left\{ L\right\}\frac{1}{\mathbb{F}(d+f)^{4}},$$ in
$D_r(q),$ where $\overline{D_{r}(q)}\subset \bar\Delta,$ and
$$
\gamma=\sum (\xi_i-\xi_i(q))^2,\;\;\;\;\; L = \epsilon \frac{\sum
(x^{k})^2}{(d+f)^{2}},
$$  and $\epsilon$ is a positive constants to be determined later. Since $d_{E}(p,\p D_{r}(q))\leq d_{E}(p,\p \Delta),$ by Proposition 2 in \cite{D2}, we have
$$\lim_{p\to \p D_{r}(q)}\det(u_{ij})d^2_{E}(p,\p D_{r}(q))=0.$$
Then,  $F$
attains its supremum at some interior point $p^*$ of $ D_r(q)$. At
$p^*$, we have,
\begin{equation}\label{eqn_2.4.a}
 L_i - \frac{4f_i}{d+f}-\frac{2\gamma_i}{ r^2-\gamma }- (log \mathbb{F})_i = 0,
\end{equation}
and
\begin{eqnarray}\label{eqn_2.5.a}
&&   \sum u^{ij}{L_{ij}} - 4 \frac{\sum u^{ij}f_{ij}}{d+f}
 + 4\frac{\sum u^{ij}f_if_j}{(d+f)^2}
\\&&-\frac{2\sum u^{ij}\gamma_{ij}}{ r^2-\gamma }-\frac{ 2\sum
u^{ij}\gamma_{i}\gamma_{j}}{\left(r^2-\gamma\right)^2}  + \frac{\sum u^{ij}\mathbb{F}_i\mathbb{F}_j}{\mathbb{F}^2} + A\leq 0.
\nonumber
\end{eqnarray}
We choose  $\epsilon =
\frac{1}{2000 b}$. By the same calculation in \cite{CLS5} we prove the lemma.
 $\Box$

\v

\section{Interior estimate of $\Theta$}\label{sect_3.2}

Let $\Delta\subset \mathbb R^2$ be the Delant polytope.
The purpose of this section is to estimate  $\Theta$ near the boundary $\partial \Delta.$
We use the blow-up analysis to prove the following result.
\begin{theorem}\label{theorem_5.1} Let $u\in \mc C^\infty(\Delta,v)$. Choose a coordinate system $(\xi_1,\xi_2)$ such that $\ell=\{\xi|\xi_1=0\}.$ Denote
 $B_b(p, \Delta)=\{ q\in  {\Delta}|\; d_{u}(q,p)< b\}, $  where $d_u(p,q)$ is the distance from $p$ to $q$
with respect to the Calabi metric $G_u$.
Let $p\in \ell^\circ$  such that $B_b(p, \Delta)$ intersection
with $\partial\Delta$ lies in the
interior of $\ell$. Suppose that
\begin{equation}\label{eqn_7.1}
\| \mathcal{S}_{\mathbb{D}}(u)\|_{C^3(B_b(p, \Delta))}\leq \mff N_{2} ,  \quad
h_{22}|_{\ell}\geq \mff N_{2} \inv,
\end{equation}
for some constant $\mff N_{2}>0,$ where $h=u|_{\ell}$ and $\|.\|_{C^3(\Delta)}$
denotes the Euclidean  $C^{3}$-norm.
Then, for any $p\in B_{b/2}(p, \Delta) $,
\begin{equation}
\left(\Theta+\mathcal K \right)(p)
d^2_u(p,\ell)\leq \mff C_5,
\end{equation}
where    $\mff C_5$ is a positive constant
depending only on   $\mff N_2$.
\end{theorem}

To prove this theorem we need the following theorems, the proofs can be found in \cite{LJ} and \cite{CLS1}.

 \begin{theorem}\label{theorem_6.4}[Li-Jia] Let $u(\xi_1,\xi_2)$ be a $C^\infty$
 strictly convex function defined in a convex
domain $ \Omega \subset \real^2$. If
 $$
\mc S(u)=0,\;\;\;
u|_{\partial\Omega}=+\infty,$$
then the graph of $u$ must be an elliptic paraboloid.
\end{theorem}
\begin{remark}\label{remark_6.5}   If we only assume that $u\in C^{5}$, Theorem \ref{theorem_6.4} remains  true.
\end{remark}
\begin{theorem}\label{theorem_5.6}
Suppose that $u\indexm$ is a sequence of smooth
strictly convex functions on $\Omega^* \subset \mathbb{R}^n$ containing 0
and $\Theta_{u\indexm}\leq\mff  N_{3}^2,$
and that $u\indexm$ are already normalized such that
$$
u\indexm\geq u\indexm(0)=0,\;\;\;
\frac{\partial^2u\indexm}{\p \xi_{i}\p \xi_{j}}(0)=\delta_{ij}.$$ Then
\begin{enumerate}
\item[(a)] there exists a constant  $\mff a>0$ such that $B_{\mff a,u_k}(0)\subset \Omega $,
\item[(b)] there exists a subsequence of $u\indexm$ that locally
$C^2$-converges to a strictly convex function $u_\infty$
in $B_{\mff a,u_\infty}(0)$,
\item[(c)] moreover,  if $(\Omega^* ,G_{u\indexm})$ is complete,
$(\Omega^* ,G_{u_\infty})$ is complete.
\end{enumerate}
\end{theorem}

\begin{theorem}\label{theorem_6.7}
Let $\Omega^*\subset \mathbb{R}^n$ be a convex domain. Let $u$ be a smooth strictly convex function on
$\Omega^*$ with  $\Theta\leq\mff N_{3}^2$. Then $
(\Omega^*,G_u)$ is
 complete if and only if  the graph of $u$ is Euclidean complete
in $\mathbb{R}^{n+1}$.
\end{theorem}

\v\n
{\bf Proof of Theorem \ref{theorem_5.1}.} If the theorem is not true, then there
exists a sequence of functions $u\indexm$ and a sequence of points
$p_k\in B_{b/2}(p_{k},\Delta)$ such that
\begin{equation}\label{eqn_6.5}
\Theta_{u\indexm}(p_k)d^2_{u\indexm}(p_k,
\ell)\to \infty.\end{equation}

Let $B\indexn$ be the $\frac{1}{2}d_{u\indexm}(p_k,
\ell)$-ball centered at $p_k$ and
consider the {\em affine transformationally  invariant function}
$$
F_k(p)=\Theta_{u\indexm}(p)d^2_{u\indexm}(p,
\partial{B}\indexn).
$$
$F_k$ attains its maximum at $p_k^*$. Put
$$d_k=\frac{1}{2}d_{u\indexm}(p^*_k,\partial{B}\indexn).$$
By adding linear functions we assume that
\[u_k(p_k^*)=0,\;\;\; \nabla u_k(p_k^*)=0.\;\;\; \]
  Let $l_{k}$ be the largest constant such that the section $S_{u_{k}}(p^{*}_{k},l_{k})$ is compact. 
Denote by $h_{k}$ the restriction of $u_{k}$ to $\ell$. Then, $h_{k}$ locally
uniformly converges to a convex function $h $ on $\ell$, and  $u_{k}$  locally   $C^6$-converges
in $\Delta$ to a strictly convex function $u_\infty$. $u_\infty$ can be
continuously  extended to be defined on $\bar \Delta.$  By Lemma 3.3 in \cite{LLS}, $C^{0}$-estimate and Lemma \ref{lemma_2.3.a} we have
  $$ - C_1- \mff    \log d_E(p,\partial \Delta)  \leq \log\det(D^2 u_k)(p)\leq C_1-   8  \log d_E(p,\partial \Delta) $$
 and 
\begin{equation}\label{eqn_5.4}
   \frac{\p^2h_{k}}{\p \xi_{2}^2}\geq C_{2} 
\end{equation}
  for some positive constants $C_{1},C_{2} $ independent of $k.$ As in \cite{CLS2} we can conclude that
  \begin{equation} \label{eqn_5.5a}
 u_\infty(q) = h(q),\;\;\;\;\;\forall  q\in \ell^\circ 
 \end{equation}
Using this and   the interior regularity we have
 $\lim\limits_{k\to\infty} l_{k} =0.$
  Then it follows from \eqref{eqn_5.4} and \eqref{eqn_5.5a}  that
\begin{equation}\label{eqn 5.4}
\;\;\;\lim_{k\to\infty}diam \left(S_{u_{k}}(p^{*}_{k},l_{k})\right)=0.
\end{equation}
 By taking a proper coordinate  translation  we may assume that the coordinate of $p^*_k$ is $0$. Then
\begin{itemize}
\item $\Theta_{u\indexm}(0)d^2_{k}\to \infty$. \item
$\Theta_{u\indexm}\leq 4 \Theta_{u\indexm}(0)$ in
$B\indexn_{d_k}(0)$.
\end{itemize}
By \eqref{eqn_6.5}  we have
\begin{equation}\label{eqnc_4.2}\lim\limits_{k\to\infty}
\Theta_{u\indexm}(0)=+\infty.\end{equation}
We take an affine transformation   on $u_k$:
$$\xi^{\star k}=A_k\xi,\;\;u^{\star}_k(\xi^{\star k}):= \lambda_k
u_k\left(A_k\inv \xi^{\star k}\right),$$ where $\lambda_k = \Theta_{u\indexm}(0).$ Choose $A_k$ such that  $\frac{\partial^2 u^{\star}_{k}}{\p \xi_{i}\p \xi_{j}}(0)=\delta_{ij}.$ Denote $A_k\inv=(b^k_{ij})$. Then by the affine transformation
rule we know that
$\Theta_{u^{\star}_k}(0) = 1,$
and for any fixed large $R$, when $k$ large enough
$$\Theta_{u^{\star}_k}\leq 4
\;\;\;\;in \;\;\;B\indexn_{R}(0).$$
Denote by $d_k^{\star}$ the change of $d\indexm$ after the affine transformation. Then $$\lim_{k\to\infty}d\indexm^{\star}=+\infty.$$
By Theorem \ref{theorem_5.6}, one concludes that $u^{\star}_k$ locally uniformly $C^{2}$-converges
to a function $u^\star_\infty,$ and the graph of $u_\infty^{\star}$ is complete with respect to the Calabi metric $G_{u^\star_\infty}.$ By Theorem \ref{theorem_6.7} the graph of $u_\infty^{\star}$ is Euclidean complete.
We have, in $B_{R,u^\star_{\infty}}(0),$
\begin{equation}\label{eqn 5.6}
C_{2}^{-1}\leq
\lambda_{\min}(u^\star\indexm)\leq \lambda_{\max}(u^\star\indexm)\leq C_{2},
\end{equation}
where $C_{2}>0$ is a constant depending only on $R.$ Here $\lambda_{\min}(u^\star\indexm)$ and $\lambda_{\max}(u^\star\indexm)$ denotes the minimal and maximal eigenvlues of the Hessian of  $u^\star_{k}$ in $B_{R,u^\star_{\infty}}(0).$
Then  there exists an Euclidean ball $D_\epsilon(0)$ such that  $D_\epsilon(0)\subset A_k(S_{u_{k}}(0,l_{k}))$ when $k$ large enough. Therefore
 $A_k\inv D_\epsilon(0)\subset S_{u_{k}}(0,l_{k}).$ It follows from \eqref{eqn 5.4} that for any $1 \leq i,j\leq 2$
\begin{equation}\label{eqna_4.1}\lim_{k\to \infty}\left|\frac{\partial \xi_i}{\partial \xi^{\star k}_j}\right|=\lim_{k\to \infty}|b^k_{ij}|= 0, \end{equation} where $ A_k\inv=(b^k_{ij}) .$
By a direct calculation we have
$$\left|\frac{\partial \mathcal{S}_{\mathbb{D}_{k}}(u^{\star}\indexm)}{\partial \xi_i^{\star k}}\right|=
 \lambda_k\inv   \left|\sum_{j}b^k_{ji}\frac{\partial \mathcal{S}_{\mathbb{D}_{k}}(u \indexm)}{\partial \xi_j}\right|  \leq 8diam(\Omega){\epsilon\inv}\lambda_k\inv K_o\to 0 $$ as $k$ goes to infinity. Similarly, we have
$$\lim_{k\to\infty}\|\mathcal{S}_{\mathbb{D}_{k}}(u^{\star}\indexm)\|_{C^2}= 0,\;\;\lim_{k\to\infty}\max |\nabla_{\xi^{\star k}} log \mathbb{D}_{k}|= 0.
$$
  By  \eqref{eqn 5.6} there is a constant $r_1>0$ independent of $k$ such that $S_{u^\star_{k}}(q,r_1)\subset B_{R,u^\star_{\infty}}(0)$ for any $q\in B_{R/2,u^\star_{\infty}}(0).$ It follows from Theorem \ref{theorem_6.1} below that
$u^{\star}\indexm$ locally  $C^{3,\alpha}$-converges to a function $u^\star_\infty.$
Then  by the standard elliptic equation technique we obtain that $ u^{\star}\indexm $ locally   $C^5$-converges to $ u^{\star}_\infty$ with
$$
  \mathbb{D}=constant,\;\;\;\; \Theta_{u^{\star}_\infty}(0)=1.
$$  Hence by Theorem \ref{theorem_6.4} and Remark \ref{remark_6.5},
$u^{\star}_\infty$ must
be quadratic and $\Theta\equiv 0$. We get a contradiction. $\blacksquare$
\v
By the same argument of \cite{CLS1} we have the following corollaries. 
\begin{corollary}\label{corollaryc_4.2.1} 
Let $u\in \mc C^\infty(\Delta,v)$. Suppose that  
\begin{equation}\label{eqn 5.8a}
\max_{\Delta} u-\min_{\Delta} u\leq \mff N_2,\;\;\;\;\;
\| \mathcal{S}_{\mathbb{D}}(u)\|_{C^3 (  \Delta) }\leq \mff N_{2}  
\end{equation}
for some constant $\mff N_{2}>0.$ 
Then, for any $p\in  \Delta$,
\begin{equation}
\left(\Theta+\mathcal K \right)(p)
d^2_u(p,\p\Delta)\leq \mff C_5,
\end{equation}
where    $\mff C_5$ is a positive constant
depending only on   $\mff N_2$.
\end{corollary}

\begin{corollary}Let $u$ be as that in Corollary \ref{corollaryc_4.2.1}, and $\mathbb A(u)$ be operator such that 
$\mathbb A(u)d^{2}(p,\p\Delta)$ is affine invariant. 
Suppose that  $$
\|\mathbb A(u)\|_{C^{3}(\Delta)}\leq \mff N_{2},\;\;\;\; 
|\mathbb A(u) |\geq \delta>0.$$
Then
there is a constant $\mff C_5>0$, depending only on $\Delta$ and $\mff N_2$, such that
\begin{equation}\label{eqn_4.4b}
\|\nabla\log|\mathbb A(u) |\|^2(p) d^2_u(p,\partial {\Delta})\leq \mff C_5,\;\;\; \forall p\in \Omega. \end{equation}
\end{corollary}

\v

\section{A convergence theorem and its application}\label{sect_3.2}

\v
\subsection{\bf A convergence theorem}\label{sect_3.2}
\v
In this section we extend the Theorem 3.6 in \cite{CLS1} to the generalized Abreu equation. Denote by $\mathcal{F}(\Omega^*,C)$ the class of smooth convex functions
defined on $\Omega^*$ such that
$$ \inf_{\Omega^*} {u } = 0,\;\;\;
u= C>0\;\;on\;\;\partial \Omega^*.$$
The main result of this subsection is the following convergence
theorem.
\begin{theorem}\label{theorem_6.1}
Let $\Omega^*\subset \mathbb{R}^2$ be a normalized convex domain. Let $u\indexm\in \mc
F(\Omega^*,1)$ be a sequence of functions and $p^o\indexm$ be
the minimal point of $u\indexm$. Let $\mathbb D_{k}>0$ be given smooth function defined on $\overline{\Omega}^*.$
Suppose that there is $\mff N_{4}>0$ such that
$$|\mathcal{S}_{\mathbb{D}_{k}}(u\indexm)|\leq \mff N_{4},\;\;\;\;\mff N_{4}^{-1} \leq \mathbb D_{k}\leq \mff N_{4}$$
and
$$
 \sup_{\Omega^*}|\nabla_{\xi}\log \mathbb D_{k}|\leq \mff N_{4}.
$$
Then there exists a subsequence
of functions, without loss of generality, still denoted by
$u\indexm$, locally uniformly  converging to a function $u_\infty$ { in $\Omega^*$} and
$p_{k}^o$ converging to $p^o_\infty$ such that $$d_E(p^o_{\infty},\partial\Omega^*)>\mff s$$
for some constant $\mff s>0$ and in
$D_\mff s(p_{\infty}^o)$
$$
\|u\|_{C^{3,\alpha}}\leq\mff  C_6
$$
for some  $\mff C_6>0$ and $\alpha\in (0,1)$.
\end{theorem}

To prove Theorem \ref{theorem_6.1},
we need the following lemma, the proof can be find in \cite{LLS}.

\begin{theorem}\label{theorem_6.2}
Let $\Omega\subset \mathbb R^{n}$ be a bounded normalized convex domain. Let $f_{k}\in \mc
F(\Omega,C)$ be a sequence of functions satisfying the equation $\mathcal{S}_{\mathbb{D}_{k}}(f_{k})=A_{k}$.
Suppose that $A_{k}$(resp. $\mathbb D_{k}$) $C^{m}$-converges to $A$ (resp. $\mathbb D>0$)  on $\bar\Omega$, and there are constants $0< \mff N_{5}<\mff N_{6}$ independent of $k$ such that
\begin{equation} \label{equ 4.5}
\mff N_{5}\leq det\left(\frac{\p^2 f_{k}}{\p x^{i}\p x^{j}}\right)\leq \mff N_{6}\end{equation}
hold in $\Omega$. Then there exists a subsequence
of functions, without loss of generality, still denoted by
$f_{k}$,   locally uniformly  converging  to a function $f_\infty$ { in $\Omega$} and, for any open set $\Omega_o$ with $\bar{\Omega}_o\subset \Omega$, and for any $\alpha\in (0,1)$,  $f_{k}$ $C^{m+3,\alpha}$-converges
to $f_\infty$ in $\Omega_o$.
\end{theorem}
\v\n
{\bf Proof of Theorem \ref{theorem_6.1}} By Lemma \ref{lemma_4.2} we have uniform estimate in $\Omega^\circ_k:=\{\xi | u_k\leq \frac{1}{2}\}$:
$$\frac{\det(u_{ij})}{(d+f)^{4}}\leq e^8\mff{C}_3$$
for any one $u$ in $\{u_k\}$. Then there is $r>0$ independent of $k$ such that $D_r(0)\subset \nabla u_k(\Omega^\circ_k)$, and there is a constant $d>0$ such that
\begin{equation}\label{eqn_5.5}
\det(f_{ij})\geq \frac{1}{e^8\mff{C}_3(r+d)^{4}},\;\;\;-1\leq f\leq r,\;\;\;\end{equation}
for any one $f$ in $f_k$. Then there exists a subsequence
of functions, without loss of generality, still denoted by
$f_{k}$, locally uniformly   converging  to a function $f_\infty$ in $D_r(0)$. By \eqref{eqn_5.5} and the Alexandrov-Pogorelov theorem $f_\infty$ is strictly convex. Then there is a constant $C_3>0$ such that the section
$$\bar{S}_{f_{\infty}}(0,C_3):=\{x |f_{\infty}\leq C_3\}$$
is compact. By Lemma \ref{lemma_4.1} we have, in the section $\bar{S}_{f_{\infty}}(0,C_3/2)$,
\begin{equation}\label{eqn_5.5}
\det(f_{ij})\leq e^8\mff{C}_2\end{equation}
for any one $f$ in $f_k$. Using Theorem \ref{theorem_6.2} we conclude that $f_{k}$ $C^{3,\alpha}$-converges to $f_{\infty}$ in the section $S_{f_{\infty}}(0,C_3/2)$. It follows that $p^o\indexm$ converges to $p^o_\infty$ and there exist constants $\mff s$ such that $d_E(p\indexm^o,\partial\Omega)>2\mff s$, and $u\indexm$ $C^{3,\alpha}$-converges to $u_{\infty}$
in $D_{\mff s}(p^o_\infty)$. The theorem follows.  $\blacksquare$
\v
\subsection{\bf An application of Theorem \ref{theorem_5.1}}\label{subsect_3.2}
\v

In order to use affine blow-up technique we need to control sections (see Subsection \S7.2\cite{CLS1} and Subsection \S5.3-\S5.4 \cite{CLS2}. For the generalized Abreu equation we need more arguments due to the presence of the function
$\mathbb D$.  The following Theorem \ref{theorem_6.3} will be used in our next works.
\v
Let $\Delta\subset \mathbb R^{2}$ be a Delzant ploytope. Let $\ell$ be an edge of $\Delta$ and
$\ell^\circ$ be the interior of $\ell$.  Let $q\in \ell^{\circ}$. We fix a coordinate system on
$\t^\ast$ such that { (i) $\ell=\{\xi | \xi_1=0\}$,
(ii) $\Delta\subset \{\xi_1>0\}$.}
\v
Let $u\indexm \in \mc C^\infty(\Delta,v)$ be a sequence of
functions with $\mc S_{\mathbb{D}}(u\indexm)=A\indexm$,   let $p^{\circ}_k$ be points such that
\begin{equation}\label{eqn 5.5}
d_{u\indexm}(p^{\circ}_k, \p \Delta)=d_{u\indexm}(p^{\circ}_k, \ell^\circ)\rightarrow 0,
 \end{equation}
 where $d_{u_k}(p^{\circ}_k,\p \Delta)$ is the geodesic distance  from $p^{\circ}_k$ to $\p \Delta$ with respect to the Calabi metric $G_{u_k}.$ Suppose that $A\indexm$   $C^3$-converges to $A$ on $\bar \Delta$.
By the interior regularity (see \cite{LLS}) we have
$u_{k}$ locally uniformly $C^{6,\alpha}$ converges to a strictly convex function $u_{\infty}$ and
$\lim_{k\to \infty} d_{E}(p^{\circ}_k, \ell^\circ)=0.$
 Then  by the $C^{0}$-estimate $u_{k}|_{\ell^{o}}$ locally uniformly to a  convex function $h.$

Let $u$ be one of $u\indexm$.  By adding a linear function we normalize $u$ such that $p^\circ$ is the minimal point of $u$; i.e.,
 \begin{equation}\label{eqnc_7.4}
  u(p^\circ)=\inf u.
 \end{equation}
Let $\check p$ be the minimal point of $u$ on $\ell$.
By a coordinate translation and by adding some constant to $u$, we may require that
\begin{equation}\label{eqnc_7.5}
u(p^{\circ})=0,\;\xi(p^{\circ})=0
\end{equation}
As in the proof of Theorem \ref{theorem_5.1} we have 
\begin{equation}\label{eqn 5.8}
\;\;\;\lim_{k\to\infty}diam \left(S_{u_{k}}(p^{\circ}_{k},u_{k}(\check{p}_{k})-u_{k}(p^\circ_{k})\right)=0.
\end{equation}

We consider the following
affine transformation on $u$:
\begin{equation}\label{eqn_7.1a}
\tilde u(\tilde{\xi})=\lambda u(A\inv(\tilde{\xi})),
\end{equation}
where $ A(\tilde{\xi}_1, \tilde{\xi}_2)=\sum a_i^j \tilde{\xi}_j.$ We choose $\lambda$ and $(a_i^j)$ such that, at $p^{\circ}$,
\begin{equation}\label{eqn_7.1a}
\tilde{u}_{ij}(0)=\delta_{ij}, \;\; d_{\tilde{u}}(p^{\circ}, \p \Delta)=1.
\end{equation}
Then we have
a sequence of functions $\tilde{u}_k$ with $\tilde{A}_k\rightarrow 0$. Denote $\mathcal{H}_k:=\tilde{u}_k(\check p_k)-\tilde{u}_k(p^\circ_k)$.

\begin{theorem}\label{theorem_6.3}
Let $\tilde{u}_k$ be a sequence of functions satisfying $\mc S_{\mathbb{D}}(u\indexm)=A\indexm$ as above. Suppose that
$$\tilde{u}_{ij}(0)=\delta_{ij}, \;\; d_{\tilde{u}}(p^{\circ}, \p \Delta)=1,\;\;\tilde{u}_k(p_k^{\circ})=0,\;\xi(p_k^{\circ})=0.$$
Then there are constants $\mff C_{8}>\mff C_{7}>0$ independent of $k$ such that
$$\mff C_{7}\leq \mathcal{H}_k\leq \mff C_{8}.$$
\end{theorem}
\v\n
{\bf Proof.} We first prove $\mathcal{H}_k\geq \mff{C}_7$. By Theorem \ref{theorem_5.1} we have $$\Theta_k\leq 4\mff C_5,\;\;\;\;\; in\;
B_\half(p_k^\circ).$$  Then by Theorem \ref{theorem_5.6} we conclude that
$u_k(\check{p}_k)-u_k(p^\circ_k)\geq \mff{C}_7$.
\v

 Now we prove $\mathcal{H}_k \leq \mff{C}_8$.   Suppose that  $ u\indexm(\check p \indexm) $ has no upper
bound. Then we can choose a sequence  of constants $N\indexm\to
\infty$ such that
$$
 0<N\indexm<u\indexm(\check p \indexm),\;\;
\lim_{k\to \infty} N\indexm \max|\tilde A_{k}|=0.  $$ For each
$u\indexm$ we take an affine transformation $\hat A\indexm:=(A\indexm,
(N\indexm)\inv)$ to get a new function $\hat{u}\indexm$, i.e,
$$
\hat  u\indexm= (N\indexm)\inv\tilde  u\indexm\circ (A\indexm)\inv.
$$
such that $S_{\hat u\indexm}(A\indexm p\indexm^\circ, 1)$  is normalized. Then from \eqref{eqn 5.8} we conclude that
$$
\lim_{k\to\infty}\sup_{S_{\hat u\indexm}(A\indexm p\indexm^\circ, 1)}|\nabla \log \mathbb D_{k}|\rightarrow 0.
$$
Obviously,
$$ \lim_{k\to \infty}\hat A_{k} =
\lim_{k\to \infty}N\indexm\tilde A_{k}= 0,$$
$$\lim_{k\to\infty} d_{\hat{u}\indexm}(A\indexm p^\circ\indexm,
A\indexm\ell) =\lim_{k\to\infty}\frac{1}{\sqrt{N\indexm}}
 d_{u\indexm}(p^\circ,\ell)=0. $$
On the other hand,  by Theorem  \ref{theorem_5.1} we conclude that
 $\hat{u}\indexm$ locally  $C^{3,\alpha}$-converges to a
  strictly convex function $\hat{u}_\infty$  in a neighborhood of the minimal point of $\hat{u}_\infty$.
In particular, there is a constant $C_3>0$ such that
  $$
d_{\hat{u}\indexm}(A\indexm p^\circ\indexm, A\indexm\ell)\geq
C_3.  $$
 We get a contradiction.
Hence, we prove the lemma. $\blacksquare$

\v\v\v\v


\begin{thebibliography}{00}

\bibitem{Abreu1998} M. Abreu, \emph{K\"ahler geometry of toric varieties and extremal metrics},
Internat. J. Math., 9(1998), 641-651.

\bibitem{ACGF}V. Apostolov, D. Calderbank, P. Gauduchon, C. T$\rm\varnothing$nnesen-Friedman, \emph{Extremal K\"ahler metrics on projective bundles over a curve,}   Adv. Math. 227 (2011), no. 6, 2385¨C2424.

\bibitem{AM} V. Apostolov, G. Maschler, \emph{Conformally K\"ahler, Einstein-Maxwell Geometry}, arXiv:1512.06391. 

\bibitem{C1} L. A. Caffarelli, \emph{Interior $W^{2,p}$ estimates for solutions of
Monge-Amp\`{e}re equations}, Ann. Math.,
131(1990), 135-150.

\bibitem{C-G}  L. A. Caffarelli, C. E Guti\'{e}rrez, \emph{Properties of the
solutions of the linearized Monge-Amp\`{e}re equations},
Amer. J. Math., 119(1997), 423-465.

\bibitem{CHLLS} B. Chen, Q. Han, A.-M. Li, Z. Lian, L. Sheng, Prescribed scaler curvatures for homogeneous toric bundles,
Preprint.

\bibitem{CLS5} B. Chen, A.-M. Li, L. Sheng,
\emph{The Abreu equation with degenerated boundary conditions},
J. Diff. Equations, 252(2012), 5235-5259.

\bibitem{CLS3}
B. Chen, A.-M. Li, L. Sheng,
\emph{Interior regularity on the Abreu equation},
Acta Mathematica Sinica, 29(2013), 33-38.

\bibitem{CLS1}
B. Chen, A.-M. Li, L. Sheng,
\emph{Affine techniques on extremal metrics on toric surfaces},
 arXiv:1008.2606.

\bibitem{CLS2}
B. Chen, A.-M. Li, L. Sheng,
\emph{Extremal metrics on toric surfaces},
 arXiv:1008.2607.



\bibitem{CLS4} B. Chen, A.-M. Li, L. Sheng,
\emph{Uniform $K$-stability for extremal metrics on toric varieties}, arXiv:1109.5228v2.


\bibitem{D1} S. K. Donaldson, \emph{Scalar curvature and stability
of toric varieties}. J. Diff. Geom.,  62(2002), 289-349.

\bibitem{D2} S. K. Donaldson, \emph{Interior estimates for solutions of Abreu's equation},
Collect. Math., 56(2005), 103-142.

\bibitem{D3} S. K. Donaldson, \emph{Extremal metrics on toric surfaces: a continuity method},
J. Diff. Geom., 79(2008), 389-432

\bibitem{D4} S. K. Donaldson, \emph{Constant scalar curvature metrics on toric surfaces},
Geom. Funct. Anal., 19(2009), 83-136.

\bibitem{D5} S.K. Donaldson, {\em K\"ahler geometry on toric manifolds, and some other manifolds with large symmetry.} Handbook of Geometric Analysis, No. 1, 2008.

 

\bibitem{FengSzeke} R. Feng, G. Sz\'ekelyhidi, \emph{Periodic solutions of Abreu's equation},
Matt. Res. Lett., 18(2011), 1271-1279.

\bibitem{Guillemin1994} V. Guillemin, \emph{K\"ahler structures on toric varieties},
J. Diff. Geom., 40(1994), 285-309.


\bibitem{LJSX} A.-M. Li,  R. Xu,  U. Simon, F. Jia, \emph{Affine Bernstein
Problems and Monge-Amp\`{e}re Equations}, World Scientific, 2010.

\bibitem{LJ} A.-M. Li, F. Jia, \emph{A Bernstein properties of some fourth order
partial differential equations}, Result. Math., 56 (2009), 109-139.

\bibitem{LLS}  A.-M. Li, Z. Lian, L. Sheng, {\em Interior regularity for the generalized Abreu equation}, Preprint.


\bibitem{N-1}  T. Nyberg,  Constant Scalar Curvature of Toric Fibrations. PhD
thesis.

\bibitem{PS} Podesta, Spiro, Kahler-Ricci solitons on homogeneous toric bundles I, II Arxiv DG/0604070/0604071

\bibitem{R}  A. Raza. Scalar curvature and multiplicity-free actions. PhD thesis.





\bibitem{Sz} G. Szi\'ekelyhidi, \emph{Extremal metrics and K-stability},
Bull. London Math. Soc.,  39(2007), 76-84.


\bibitem{T2} G. Tian, \emph{Canonical Metrics in K\"ahler Geometry},
Lectures in Mathematics ETH
Zurich,  Birkh\"auser Verlag, Basel, 2000.



\bibitem{TW2002} N. S. Trudinger, X. Wang, \emph{Berstein-J\"orgens theorem
for a fourth order partial differential equation}, J. Partial Diff. Equations, 15(2002),
78-88.








 \end{thebibliography}
 \end{document}